\documentclass[a4paper,12pt]{article}

\usepackage[left=2cm,right=2cm, top=2cm,bottom=2cm,bindingoffset=0cm]{geometry}

\usepackage{verbatim}
\usepackage{amsmath}
\usepackage{amsthm}
\usepackage{amssymb}
\usepackage{delarray}
\usepackage{cite}
\usepackage{hyperref}

\newcommand{\al}{\alpha}

\newcommand{\ga}{\gamma}
\newcommand{\de}{\delta}

\newcommand{\vv}{\varphi}
\newcommand{\iy}{\infty}

\newtheorem{thm}{Theorem}
\newtheorem{lem}{Lemma}

 \allowdisplaybreaks

\begin{document}

\begin{center}
{\large\bf
An inverse problem for the quadratic pencil of non-self-adjoint 

matrix operators on the half-line}
\\[0.2cm]
{\bf Natalia Bondarenko, Gerhard Freiling} \\[0.2cm]
\end{center}

\vspace{0.5cm}

{\bf Abstract.} We consider a pencil of non-self-adjoint matrix Sturm-Liouville
operators on the half line and study the inverse problem of constructing this pencil 
by its Weyl matrix. A uniqueness theorem is proved, and a constructive algorithm for the solution
is obtained.

\medskip


\vspace{1cm}

{\large \bf 1. Introduction} \\

In this paper, we consider the pencil $L = L(\ell_{\rho}, U_{\rho})$ given by the differential expression
\begin{equation} \label{eqv}
    \ell_{\rho}(Y) := Y'' + (\rho^2 I + 2 i \rho Q_1(x) + Q_0(x)) Y, \quad x > 0, 
\end{equation}
with the initial condition
\begin{equation} \label{BC}
    U_{\rho}(Y) := Y'(0) + (i \rho h_1 + h_0) Y(0) = 0.
\end{equation}
Here $Y(x) = [y_k(x)]_{k = \overline{1, m}}$ is a column vector, $\rho$ is the spectral parameter,
$I$ is the $m \times m$ unit matrix,
$Q_s(x) = [Q_{s, jk}(x)]_{j,k = \overline{1, m}}$ are $m \times m$ matrix-functions,
$h_s = [h_{s, j k}]_{j, k = \overline{1, m}}$, 
where $h_{s, j k}$ are complex numbers.

We assume that $\det (I \pm h_1) \ne 0$.
This condition excludes problems of Regge type (see \cite{Yur84}) from consideration,
as they require a separate investigation.  

Differential equations with nonlinear dependence on the spectral parameter,
or with so-called ``energy--dependent'' coefficients, frequently appear in
mathematics and applications (see \cite{Kel71, KS83, Shkal83, Markus86, Yur97} and references therein).
The pencil $L$ is a natural generalization of the scalar pencil ($m = 1$) of Sturm-Liouville operators.
Inverse spectral problems consist in recovering operators from their spectral characteristics.
In the scalar case, inverse problems for quadratic pencils were studied
in the works \cite{JJ72, JJ76, Yur00} (half-line), \cite{MG86, MG89, Kam08} (full line) and 
\cite{GG81, BU06, BU12, HP12, Pron12} (finite interval).

The goal of this paper is to present a solution of the inverse problem for the matrix 
differential pencil $L$. Note that we consider the pencil in the most general 
non-self-adjoint case.
We apply and extend the ideas, developed for the matrix Sturm--Liouville equation
(see \cite{AM60, Yur06, FY07, CK09, MT10, Bond11, Bond12}), to the problem with 
nonlinear dependence of the spectral parameter.

Now we provide some notations and definitions necessary for formulating the inverse problem.
We use the notation $\mathcal{A}(\mathcal{I}; \mathbb{C}^{m \times m})$ for a class of the matrix functions
$F(x) = [f_{jk}(x)]_{k = \overline{1, m}}$ with entries $f_{jk}(x)$ belonging to the class $\mathcal{A}(\mathcal{I})$
of scalar functions. The symbol $\mathcal{I}$ stands for an interval or a segment. 
We say that the pencil $L$ belongs to the class $\mathcal{V}$, if
the entries of $Q_1(x)$ are absolutely continuous on $[0, T]$ for each $T > 0$ and 
$Q_0(x), Q_1(x), Q'_1(x) \in L((0, \infty); \mathbb{C}^{m \times m})$. Below we always consider
pencils from $\mathcal{V}$.

Denote $\Pi_{\pm} = \{ \rho \colon \pm \mbox{Im}\, \rho > 0\}$.
Let $\Phi(x, \rho) = [\Phi_{jk}(x, \rho)]_{j, k = \overline{1, m}}$ 
be the matrix solution of equation $\ell_{\rho}(Y) = 0$ satisfying the conditions
$U_{\rho}(\Phi) = I$, 
$\Phi(x, \rho) = O(\exp(\pm i \rho x))$ as $x \to \iy$ for $\rho \in \Pi_{\pm}$.
We call $\Phi(x, \rho)$ the
\textit{Weyl solution} for the pencil $L$. Put $M(\rho) := \Phi(0, \rho)$. 
The matrix $M(\rho) = [M_{jk}(\rho)]_{j,k = \overline{1, m}}$ is called
the \textit{Weyl matrix} for $L$. The notion of the Weyl matrix
is a generalization of the notion of the Weyl function ($m$-function)
for the scalar case (see \cite{Mar77, FY01}) and the notion of the 
Weyl matrix for the matrix Sturm--Liouville operator (see \cite{FY07}).

We study the following

\medskip

{\bf Inverse Problem 1.} 
{\it Given a Weyl matrix $M(\rho)$, find the coefficients of the pencil $L$.}

\medskip

We prove a uniqueness theorem for Inverse Problem 1 and provide a constructive procedure
for recovering the pencil $L$ by its Weyl matrix. In our research, we generalize the method of
spectral mappings \cite{FY01, Yur02}. The main idea of this method is the transformation of a nonlinear
inverse spectral problem to a linear equation in a proper Banach space. The important
role is played by the contour integration in the complex plane of the spectral parameter and
by the residue theorem. In contrast to the case of the matrix Sturm-Liouville operator \cite{FY07}, 
the main equation for the pencil is not always uniquely solvable. Roughly speaking, its solvability depends 
on the coefficient $Q_1$ which is not known a priori. Therefore we build the solution ``step by step'' 
in the intervals where we can solve the main equation
(see Algorithm 2), and prove that the number of steps is finite. 

The paper is organized as follows. In Section 2, we obtain some preliminary results, and in Section 3, we prove
the uniqueness theorem for Inverse Problem 1. In Section 4, the main equation is derived and conditions of
its unique solvability are studied. In Section 5, we give a constructive procedure of recovering the pencil $L$
by the Weyl matrix. The paper also contains Appendix, where we provide a detailed construction of Jost- and 
Birkhoff-type solutions for the equation $\ell_{\rho}(Y) = 0$. 

\bigskip

{\large \bf 2. Preliminaries} 

\bigskip

In this section, we construct a 
special fundamental system of solutions for the equation $\ell_{\rho}(Y) = 0$. 
Then we investigate properties of the Weyl matrix and obtain some other preliminary results.

\medskip

\begin{thm} \label{thm:FSS}
For each $\al > 0$, there exists $\rho_{\al} > 0$ such that for $|\rho| > \rho_{\al}$, $x \ge 0$
the equation $\ell_{\rho} (Y) = 0$
has a fundamental system of solutions $E_{\pm}(x, \rho)$ satisfying the following conditions:

$(i_1)$ the matrix functions $E^{(\nu)}_{\pm}(x, \rho)$, $\nu = 0, 1$, are continuous for
$|\rho| > \rho_{\al}$, $x \ge 0$ in each half-plane $\overline{\Pi_{\pm}}$;

$(i_2)$ for each fixed $x \ge 0$ the matrix functions $E^{(\nu)}_{\pm}(x, \rho)$, $\nu = 0, 1$
are analytic for $|\rho| \ge \rho_{\al}$ in each half-plane $\Pi_{\pm}$;

$(i_3)$ as $|\rho| \to \iy$, $\rho \in \overline{\Pi_{\pm}}$, uniformly with respect to $x \ge \al$ 
we have
$$
  	E_{\pm}(x, \rho) = \exp(\pm i \rho x) \left( P_{\mp}(x) + \frac{T_{\mp}(x)}{\rho} + o(\rho^{-1}) \right),
$$
$$
 	E'_{\pm}(x, \rho) = \pm \exp(\pm i \rho x) \left( (i \rho I - Q_1(x)) P_{\mp} (x) \pm i T_{\mp}(x) + o(1) \right),
$$
where the matrix functions $P_{\pm}(x)$ and $T_{\pm}(x)$ are the solutions of the Cauchy problems
\begin{equation} \label{cauchyP}
 	P'_{\pm}(x) = \pm Q_1(x) P_{\pm}(x), \quad P_{\pm}(0) = I,
\end{equation}
\begin{equation} \label{cauchyT}
 	T'_{\pm}(x) = \pm Q_1(x) T_{\pm}(x) + \frac{1}{2 i} (Q'_1(x) \pm Q_1^2(x) \pm Q_0(x)) P_{\pm}(x), 
 	\quad \lim_{x \to \iy} T_{\pm}(x) = 0;
\end{equation}

$(i_4)$ as $x \to \iy$, for each fixed $\rho \in \Pi_{\pm}$, $|\rho| > \rho_{\al}$, we have
$$
 	E_{\pm}^{(\nu)}(x, \rho) = (\pm i \rho)^{\nu} \exp(\pm i \rho x) (P_{\mp}(x) + o(1)).
$$

Note that $\rho_{\al} \to 0$ as $\al \to \iy$.
 
\end{thm}

\medskip

{\it Remark.} In the scalar case $(m = 1)$, we have
$$
 	P_{\pm}(x) = P^*_{\pm}(x) = \exp \left\{ \pm \int_0^x Q_1(t) \, dt \right\}.
$$

\medskip

In general, the proof of Theorem~\ref{thm:FSS} is based on the ideas from \cite{FY01}[Section 2.1].
However, it contains many technical difficulties. For convenience of the reader, we provide the proof in Appendix.

Now we study some important properties of
the coefficient matrices $P_{\pm}(x)$. 
Analogously to $P_{\pm}(x)$, we introduce the matrix functions $P^*_{\pm}(x)$ as the solutions
of the following Cauchy problems
\begin{equation} \label{cauchyP*}
 	{P^*}'_{\pm}(x) = \pm P^*_{\pm}(x) Q_1(x), \quad P^*_{\pm}(0) = I.	
\end{equation}

Below we mean
by $\| . \|$ the following matrix norm: $\| A \| = \max_{j = \overline{1, m}} \sum_{k = 1}^m |a_{ij}|$,
$A = [a_{jk}]_{j, k = \overline{1, m}}$.

\medskip

\begin{lem} \label{lem:P}
The following relations hold
\begin{equation} \label{relP}
 	P_{+}(x) P_{-}^*(x) = P_{-}^*(x) P_{+}(x) = I, \quad P_{-}(x) P_{+}^*(x) = P_{+}^*(x) P_{-}(x) = I, 
\end{equation}
\begin{equation} \label{estP}
	\| P_{\pm}(x) \|, \, \| P^{-1}_{\pm}(x) \| \le \exp\left( \int_0^x \| Q_1(t) \| \, dt \right) < K < \infty,
\end{equation}
for all $x > 0$. Here $K$ is some constant depending on $Q_1$.

\end{lem}

\begin{proof}
Using \eqref{cauchyP} and \eqref{cauchyP*}, we get
$$
 	(P_{-}^*(x) P_{+}(x))' = {P_{-}^*}'(x) P_{+}(x) + P_{-}^{*}(x) P_{+}'(x) =
 	- P_{-}^*(x) Q_1(x)  P_{+}(x) + P_{-}^*(x) Q_1(x)  P_{+}(x) = 0.
$$ 
Hence $P_{-}^*(x) P_{+}(x)$ does not depend on $x$, so $P_{-}^*(x) P_{+}(x) = P_{-}^*(0) P_{+}(0) = I$.
The other relations of \eqref{relP} can be proved similarly.

The estimates \eqref{estP} can be easily obtained from \eqref{cauchyP} and \eqref{cauchyP*} by 
Gronwall's lemma.
\end{proof}

\medskip

Let $S(x, \rho) = [S_{jk}(x, \rho)]_{j, k = \overline{1, m}}$ and
$\vv(x, \rho) = [\vv_{jk}(x, \rho)]_{j, k = \overline{1, m}}$
be the matrix solutions of equation $\ell_{\rho}(Y) = 0$
under the initial conditions
$$
 	S(0, \rho) = U_{\rho}(\vv) = 0, \quad S'(0, \rho) = \vv(0, \rho) = I.
$$

One can easily show that
\begin{equation} \label{PhiexpE}
 	\Phi(x, \rho) = E_{\pm}(x, \rho) (U_{\rho}(E_{\pm}))^{-1}, \quad \rho \in \Pi_{\pm},
\end{equation}
\begin{equation} \label{Phiexp}
 	\Phi(x, \rho) = S(x, \rho) + \vv(x, \rho) M(\rho).
\end{equation}


\begin{lem} \label{lem:asympt}
As $|\rho| \to \iy$, $\rho \in \overline{\Pi_{\pm}}$, the following relations hold uniformly with
respect to $x \ge 0$:
\begin{multline} \label{asymptphi}
 	\vv^{(\nu)}(x, \rho) = \frac{(i \rho)^{\nu}}{2} \exp(i \rho x) (P_{-}(x) (I - h_1) + O(\rho^{-1})) + \\ + 
 	\frac{(- i \rho)^{\nu}}{2} \exp(- i \rho x) (P_{+}(x) (I + h_1) + O(\rho^{-1})), \quad
 	\nu = 0, 1,
\end{multline}
$$
 	\Phi^{(\nu)}(x, \rho) = (\pm i \rho)^{\nu - 1} (P_{\mp}(x) (I \pm h_1)^{-1} + O(\rho^{-1})),
$$
\end{lem}

\begin{proof}
Expand the solution $\vv(x, \rho)$ by the fundamental system from Theorem~\ref{thm:FSS}:
$$
 	\vv^{(\nu)}(x, \rho) = E^{(\nu)}_{+}(x, \rho) C_{+}(\rho) + E^{(\nu)}_{-}(x, \rho) C_{-}(\rho), \quad \nu = 0, 1.
$$
Find the asymptotics for the coefficients $C_{\pm}(\rho)$, using the initial conditions for $\vv(x, \rho)$
and $(i_3)$ of Theorem~\ref{thm:FSS}:
$ C_{\pm}(\rho) = \frac{1}{2} (I \mp h_1) + O(\rho^{-1})$, $|\rho| \to \infty$. Then we immediately obtain \eqref{asymptphi}.

In order to derive the asymptotic representation for $\Phi(x, \rho)$, substitute $(i_3)$ of Theorem~\ref{thm:FSS} into \eqref{PhiexpE}. 
\end{proof}

\begin{lem} \label{lem:asymptM}
The matrix function $M(\rho)$ is analytic in each half-plane $\Pi_{\pm}$ outside a countable
bounded set of poles $\Lambda'_{\pm}$, and it is continuous in $\overline{\Pi_{\pm}}$ 
outside a bounded set $\Lambda_{\pm}$. As $|\rho| \to \iy$ the following relation holds
$$
 	M(\rho) = (i \rho)^{-1} (h_1 \pm I)^{-1} + (i \rho)^{-2} (I \pm h_1)^{-1} (\mp Q_1(0) - h_0) (I \pm h_1)^{-1} + 
 	\frac{\kappa(\rho)}{\rho^2}, \quad \kappa(\rho) = o(1), 
$$
Moreover,
$\kappa \in L_2((0, +\infty); \mathbb{C}^{m \times m})$, 
if $Q_s^{(s)} \in L_2((0, \infty); \mathbb{C}^{m \times m})$, $s = 0, 1$.
\end{lem}

\begin{proof}
Let us prove the lemma for $\Pi_{+}$. The case of $\Pi_{-}$ can be considered analogously.
By Theorem~\ref{thm:FSS}, the matrix function $M(\rho) = E_{+}(0, \rho) (U(E_{+}))^{-1}$
is analytic in $\Pi_{+}$ except in 	the zeros of $\det U(E_{+})$. By virtue of Theorem~\ref{thm:FSS}
this determinant is analytic in $\Pi_{+}$ and 
$$
	\det U(E_{+}) = (i \rho)^m \det(I + h_1) + O(\rho^{m - 1}), \quad |\rho| \to \infty.
$$
Note that we assume $\det(I + h_1) \ne 0$. Consequently, $\det U(E_{+})$ has a countable bounded number of zeros
$\Lambda'_{+}$. Moreover, $M(\rho)$ is continuous in $\overline{\Pi_{+}}$ outside the bounded set of zeros $\Lambda_{+}$ of $\det U(E_{+})$.

Using $(i_3)$ of Theorem~\ref{thm:FSS}, we derive the asymptotic formula for $M(\rho)$:
\begin{multline*}
M(\rho) = (I + \rho^{-1} T_{-}(0) + o(\rho^{-1})) \bigl[ i \rho (I + h_1) - Q_1(0) + i (I + h_1) T_{-}(0) + h_0 + o(1) \bigr]^{-1} \\
= (i \rho)^{-1} (I + \rho^{-1} T_{-}(0) + o(\rho^{-1}))
\bigl[ I - (I + h_1)^{-1}(i (Q_1(0) - h_0) + (I + h_1) T_{-}(0))  + o(1) \bigr] (I + h_1)^{-1} \\
= (i \rho)^{-1} (I + h_1)^{-1} + (i \rho)^{-2} (I + h_1)^{-1} (Q_1(0) - h_0) (I + h_1)^{-1} + o(\rho^{-2}).
\end{multline*}
In order to prove that $\kappa(\rho) \in L_2$ while $Q^{(s)} \in L_2$, $s = 0, 1$, note that under these conditions
the integrals $J_2(x, \rho)$ and $J_2(a, x, \rho)$ from Steps 2 and 3 of the proof of Theorem~\ref{thm:FSS} 
in Appendix are from $L_2$, and exactly these integrals determine the behavior of $\kappa(\rho)$.
\end{proof}

\medskip

Along with $L$, we consider the pencil $L^* = (\ell^*_{\rho}, U^*_{\rho})$, where
\begin{equation} \label{lU*}
	\begin{array}{l}
    \ell^*_{\rho}(Z) := Z'' + Z (\rho^2 I + 2 i \rho Q_1(x) + Q_0(x)), \\
    U^*_{\rho}(Z) := Z'(0) + Z(0) (i \rho h_1 + h_0),
    \end{array}
\end{equation}
Denote
$$
    \langle Z, Y \rangle = Z'(x) Y(x) - Z(x) Y'(x),
$$
where $Z = [z_k]^T_{k = \overline{1, m}}$ is a row vector ($T$ is the sign for the transposition). Then
\begin{equation} \label{smEq2}
    \langle Z, Y \rangle_{|x = 0} = U^*_{\rho}(Z) Y(0) - Z(0) U_{\rho}(Y).
\end{equation}
If $Y(x, \rho)$ and $Z(x, \rho)$ satisfy the equations $\ell_{\rho} (Y(x, \rho)) = 0$ and $\ell^*_{\rho} (Z(x, \rho)) = 0$
respectively, then
\begin{equation} \label{smEq1}
    \frac{d}{dx} \langle Z(x, \rho), Y(x, \rho) \rangle = 0.
\end{equation} 

Let $\vv^*(x, \rho)$, $S^*(x, \rho)$ and $\Phi^*(x, \rho)$ be the matrices, satisfying 
the equation $\ell^*_{\rho}(Z) = 0$ and the conditions 
$\vv^*(0, \rho) = {S^*}'(0, \rho) = U^*_{\rho}(\Phi^*) = I$, 
$U^*_{\rho}(\vv^*) = S^*(0, \rho) = 0$,
$\Phi^*(x, \rho) = O(\exp(\pm i \rho x))$, $x \to \infty$, $\rho \in \Pi_{\pm}$.
Put $M^*(\rho) := \Phi^*(0, \rho)$.
Let $E^*_{\pm}(x)$ be the matrix solutions of the equation $\ell^*_{\rho}(Z) = 0$
built analogously to $E_{\pm}(x)$ in $\Pi_{\pm}$ (see Theorem~\ref{thm:FSS}).

Then
\begin{equation} \label{PhiexpE*}
 	\Phi^*(x, \rho) = (U^*_{\rho}(E^*_{\pm}))^{-1} E^*_{\pm}(x, \rho), \quad \rho \in \Pi_{\pm},
\end{equation}
\begin{equation} \label{Phiexp*}
    \Phi^*(x, \rho) = S^*(x, \rho) + M^*(\rho) \vv^*(x, \rho).
\end{equation}
According to \eqref{smEq1}, $\langle \Phi^*(x, \rho), \Phi(x, \rho) \rangle$ does not depend on $x$.
Using \eqref{PhiexpE}, \eqref{smEq2}, \eqref{PhiexpE*} and the asymptotics ($i_4$) of Theorem~\ref{thm:FSS},
we get
$$
    \langle \Phi^*(x, \rho), \Phi(x, \rho) \rangle_{|x = 0} = M(\rho) - M^*(\rho), 
$$
$$
 	\langle \Phi^*(x, \rho), \Phi(x, \rho) \rangle_{|x = \iy} = 0.
$$
Therefore, 
\begin{equation} \label{eqMM*}
	M(\rho) \equiv M^*(\rho).
\end{equation}
Similarly,
\begin{equation} \label{wronphi}
  \langle \vv^*(x, \rho), \vv(x, \rho) \rangle = \langle \vv^*(x, \rho), \vv(x, \rho) \rangle_{|x = 0} = 0.
\end{equation}
Using \eqref{smEq2}, one can easily show that
$$
\left[\begin{array}{ll} {\Phi^*}'(x, \rho) & -\Phi^*(x,\rho)\\ -{\vv^*}'(x,\rho) & \vv^*(x,\rho) \end{array}\right]    
\left[\begin{array}{ll} \vv(x,\rho) & \Phi(x,\rho)\\ \vv'(x,\rho) & \Phi'(x,\rho) \end{array}\right] =
\left[\begin{array}{ll} I & 0 \\ 0  & I \end{array}\right].
$$
Hence,
\begin{equation} \label{inverseform}
\left[\begin{array}{ll} \vv(x,\rho) & \Phi(x,\rho)\\ \vv'(x,\rho) & \Phi'(x,\rho) \end{array}\right]^{-1} =
\left[\begin{array}{ll} {\Phi^*}'(x, \rho) & -\Phi^*(x,\rho)\\ -{\vv^*}'(x,\rho) & \vv^*(x,\rho) \end{array}\right].    
\end{equation}

\bigskip

{\large \bf 3. Uniqueness theorem} 

\bigskip

In this section, we show that the Weyl matrix $M(\rho)$ determines the 
pencil $L$ uniquely.

Along with $L$ we consider a pencil $\tilde L$ of the same form but with 
other coefficients $\tilde Q_s(x)$, $\tilde h_s$. We agree that if a symbol 
$\gamma$ denotes an object related to $L$ then $\tilde \gamma$
denotes the corresponding object related to $\tilde L$.

Consider the block-matrix $\mathcal{P}(x,\rho)=[\mathcal{P}_{jk}(x,\rho)]_{j,k=1,2}$
defined by
\begin{equation} \label{defP}
\mathcal{P}(x,\rho) \left[ \begin{array}{ll} \tilde\vv(x,\rho) & \tilde\Phi(x,\rho)\\ \tilde\vv'(x,\rho) & \tilde\Phi'(x,\rho) \end{array}\right]
= \left[\begin{array}{ll} \vv(x,\rho) & \Phi(x,\rho)\\ \vv'(x,\rho) & \Phi'(x,\rho) \end{array}\right].             
\end{equation}
Taking \eqref{inverseform} into account, we calculate
\begin{equation} \label{Pj12}
\begin{array}{l}
\mathcal{P}_{j1}(x,\rho)=\vv^{(j-1)}(x,\rho){\tilde\Phi^{*'}}(x,\rho)-
\Phi^{(j-1)}(x,\rho){\tilde\vv^{*'}}(x,\rho), \\
\mathcal{P}_{j2}(x,\rho)=\Phi^{(j-1)}(x,\rho){\tilde\vv}^{*}(x,\rho)-
\vv^{(j-1)}(x,\rho){\tilde\Phi}^{*}(x,\rho).
\end{array}
\end{equation}


\begin{lem} \label{lem:asymptP}
Assume $h_s = \tilde h_s$, $s = 0, 1$.
Then for each fixed $x > 0$, the entries of the block-matrix $\mathcal{P}$ have the following asymptotics as $|\rho| \to \infty$:
$$
  \begin{array}{ll}
   \mathcal{P}_{11}(x, \rho) = \Omega(x) + O(\rho^{-1}), & \mathcal{P}_{12}(x, \rho) = \rho^{-1} \Lambda(x) + O(\rho^{-2}), \\
   \mathcal{P}_{21}(x, \rho) = - \rho \Lambda(x) + O(1), & \mathcal{P}_{22}(x, \rho) = \Omega(x) + O(\rho^{-1}),
  \end{array}
$$                                                                                                             
where
\begin{equation} \label{defOmega}
    \Omega(x) := \frac{1}{2} \left( P_{-}(x) \tilde P^*_{+}(x) + P_{+}(x) \tilde P^*_{-}(x) \right), 
\end{equation}
\begin{equation} \label{defLambda}
    \Lambda(x) := \frac{1}{2 i} \left( P_{-}(x) \tilde P^*_{+}(x) - P_{+}(x) \tilde P^*_{-}(x) \right).
\end{equation}
\end{lem}

\begin{proof}
We prove the asymptotics for $\mathcal{P}_{11}(x, \rho)$ and $\mathcal{P}_{12}(x, \rho)$ in $\Pi_{+}$. The other 
relations can be obtained similarly. In our calculations, we will use the notation
$[I] = I + O(\rho^{-1})$, $|\rho| \to \infty$. Substitute the asymptotics $(i_3)$ of Lemma~\ref{lem:asympt} 
and the following analogous formulas
$$
 	\tilde \vv^{*(\nu)}(x, \rho) = \frac{(i \rho)^{\nu}}{2} \exp(i \rho x) (I - \tilde h_1) \tilde P_{-}^*(x)[I] +  
 	\frac{(- i \rho)^{\nu}}{2} \exp(- i \rho x) (I + \tilde h_1) \tilde P_{+}^*(x)[I], 
$$
$$
 	\tilde \Phi^{*(\nu)}(x, \rho) = (\pm i \rho)^{\nu - 1} (I \pm \tilde h_1)^{-1} \tilde P^*_{\mp}(x)[I],  \quad \nu = 0, 1,
$$
into \eqref{Pj12}:
\begin{multline*}
\mathcal{P}_{11}(x, \rho) = \bigl\{2^{-1} \exp(i \rho x) P_{-}(x) (I - h_1)[I] + 
2^{-1} \exp(- i \rho x) P_{+}(x)(I + h_1)[I]\bigr\} \\ \cdot \exp(i \rho x)(I + \tilde h_1)^{-1} \tilde P_{-}^{*}(x)[I] 
- (i \rho)^{-1} P_{-}(x)(I + h_1)^{-1}[I] \bigl\{ 2^{-1}i \rho \exp(i \rho x) (I - \tilde h_1)\tilde P_{-}^*(x)[I]
\\ - 2^{-1} i \rho \exp(- i \rho x) (I + \tilde h_1) \tilde P_{+}^*(x)[I]\bigr\} \\ = 2^{-1} \exp(2 i \rho x) P_{-}(x)
(I - \tilde h_1) (I + \tilde h_1)^{-1} \tilde P_{-}^*(x) [I]  + 2^{-1} P_{+}(x)(I + h_1) (I + \tilde h_1)^{-1} 
\tilde P_{-}^*(x)[I] \\ - 2^{-1} \exp(2 i \rho x) P_{-}(x) (I + h_1)^{-1} (I - \tilde h_1) \tilde P_{-}^*(x)[I]
+ 2^{-1} P_{-}(x) (I + h_1)^{-1} (I + \tilde h_1) \tilde P_{+}^{*}(x)[I],
\end{multline*}                                   
\begin{multline*}
\mathcal{P}_{12}(x, \rho) = (i \rho)^{-1} \exp(i \rho x) P_{-}(x) (I + h_1)^{-1}[I] \bigl\{
2^{-1} \exp(i \rho x) (I + \tilde h_1) \tilde P_{-}^*[I] \\ + 2^{-1} \exp(- i \rho x) (I + \tilde h_1) \tilde P_{+}^*(x)[I]
\bigr\} - \bigl\{ 2^{-1} \exp(i \rho x) P_{-}(x) (I - h_1)[I] \\ + 2^{-1} \exp(-i \rho x) P_{+}(x) (I + h_1) [I])
(i \rho)^{-1} \exp(i \rho x) (I + \tilde h_1)^{-1} \tilde P_{-}^*(x)[I] 
\bigr\} \\ = (2 i \rho)^{-1} \exp(2 i \rho x) P_{-}(x) (I + h_1)^{-1} (I - \tilde h_1)\tilde P_{-}^*(x)[I] +
(2 i \rho)^{-1} P_{-}(x) (I + h_1)^{-1} (I + \tilde h_1) \tilde P_{+}^*(x)[I] \\ - (2 i \rho)^{-1} P_{-}(x) (I - h_1) (I + \tilde h_1)^{-1}
\tilde P_{-}^*(x)[I] - (2 i \rho)^{-1} P_{+}(x) (I + h_1) (I + \tilde h_1)^{-1} \tilde P_{-}^*(x)[I],
\end{multline*}
for each fixed $x > 0$, $|\rho| \to \infty$.

If $h_1 = \tilde h_1$, we have
$$
 	(I - \tilde h_1) (I + \tilde h_1)^{-1} = (I + h_1)^{-1} (I - \tilde h_1). 
$$          
The exponent $\exp(2 i \rho x)$ is bounded for $\rho \in \Pi_{+}$. Consequently, we arrive at 
the required asymptotic representations.
\end{proof}
            
Now let us formulate and prove the uniqueness theorem.

\begin{thm}
If $M(\rho) = \tilde M(\rho)$, then $L = \tilde L$.
Hence the Weyl matrix determines the coefficients of the pencil \eqref{eqv}, \eqref{BC} uniquely.
\end{thm}

\begin{proof}
Substituting \eqref{Phiexp} and \eqref{Phiexp*} into \eqref{Pj12}, we get
$$
    \mathcal{P}_{11} = \vv \tilde S^{*'} - S\tilde \vv^{*'} + \vv (\tilde M^* - M) \tilde \vv^*,
$$
$$
    \mathcal{P}_{12} = S \tilde \vv^* - \vv \tilde S^{*} + \vv (M - \tilde M^*) \tilde \vv^*,    
$$
Since $M(\rho) = \tilde M(\rho) = \tilde M^*(\rho)$, for each fixed $x > 0$,
the matrix functions $\mathcal{P}_{11}(x, \rho)$
and $\mathcal{P}_{12}(x, \rho)$ are entire in $\rho$-plane. 
Hence these matrices are constant for each fixed $x > 0$. Furthermore, the assertion $M(\rho) = \tilde M(\rho)$
together with Lemma~\ref{lem:asymptM} yields $h_s = \tilde h_s$ for $s = 0, 1$, so we can use Lemma~\ref{lem:asymptP}.
Therefore, we conclude $\mathcal{P}_{11}(x, \rho) \equiv \Omega(x)$, $\mathcal{P}_{12}(x, \rho) \equiv 0$
and $\Lambda(x) \equiv 0$.

Differentiating \eqref{defOmega} and using \eqref{cauchyP}, \eqref{cauchyP*} and \eqref{defLambda}, we get
$$
    \Omega'(x) = i (\Lambda(x) \tilde Q_1(x) - Q_1(x) \Lambda(x) ) = 0.
$$
Therefore, $\Omega(x) \equiv \Omega(0) \equiv I$, $x > 0$.
Thus, $\mathcal{P}_{11}(x, \rho) = I$, $x > 0$. By virtue of~\eqref{defP} we have 
$\vv(x, \rho) \equiv \tilde \vv(x, \rho)$,
$\Phi(x, \rho) \equiv \tilde \Phi(x, \rho)$ and consequently, $L = \tilde L$.
\end{proof}

\bigskip

{\large \bf 4. Main equation of Inverse Problem 1} 

\bigskip

In this section, we assume that the Weyl matrix $M(\rho)$ of the pencil $L$ is given.
We derive a linear equation in an appropriate Banach space, called
the main equation.
This equation plays the central role in the constructive solution
of Inverse Problem 1. We investigate the solvability of the main equation.
In contrast to the case of the Sturm-Liouville operator,
the main equation for the pencil is solvable only under certain conditions on the choice of a model pencil.

\medskip

Denote
\begin{equation} \label{defD}
 	D(x, \rho, \theta) = \frac{\langle \vv^*(x, \theta), \varphi(x, \rho) \rangle}{\rho - \theta}.
\end{equation}
Using the relations $\ell_{\rho}(\vv) = 0$ and $\ell_{\theta}(\vv^*) = 0$, we derive
$$
 	\frac{d}{dx}\langle \vv^*(x, \theta), \varphi(x, \rho) \rangle = 
 	(\rho - \theta) \vv^*(x, \theta) ((\rho + \theta) I + 2 i Q_1(x)) \vv(x, \rho).
$$
Moreover,
$$
 	\langle \vv^*(x, \theta), \varphi(x, \rho) \rangle_{|x = 0} = (\rho - \theta) i h_1.	
$$
Consequently, 
$$
  	D(x, \rho, \theta) = i h_1 + \int_0^x \vv^*(s, \theta) ((\rho + \theta) I + 2 i Q_1(s)) \vv(s, \rho) \, ds.
$$
Similarly to the scalar case (see \cite{Yur00}[Lemma 3]), one can obtain the following estimate
\begin{equation} \label{estD}
 	\| D(x, \rho, \theta) \| \le C_x \frac{|\rho| + |\theta| + 1}{|\rho - \theta| + 1} 
 	\exp(|\mbox{Im}\,\rho|x) \exp(|\mbox{Im}\,\theta|x),
\end{equation}
where $C_x$ is a constant depending on $x$.

Put 
$$
 	M_{\pm}(\rho) = \lim_{z \to 0, \mbox{Re}\, z > 0} M(\rho \pm i z), \quad \rho \in \mathbb{R} \backslash \Lambda_{\pm}.
$$

Let the Weyl matrix $M(\rho)$ of the pencil $L$ be given. Using Lemma~\ref{lem:asymptM}, 
we can determine $h_1$, $h_0$ and $Q_1(0)$.
We choose a model pencil $\tilde L$ such that
for sufficiently large $\rho^* > 0$ we have
\begin{equation} \label{intM}
	\int_{|\rho| > \rho^*} \|M^{\pm}(\rho) - \tilde M^{\pm}(\rho)\|^2 \rho^4 \, d \rho < \iy.	
\end{equation}
According to Lemma~\ref{lem:asymptM}, such a model pencil can be easily chosen in the case 
$Q^{(s)}_s \in L_2((0, +\infty), \mathbb{C}^{m \times m})$, $s = 0, 1$.
It follows from \eqref{intM}, that 
\begin{equation} \label{eqhs}
	h_s = \tilde h_s, \quad s = 0, 1, 
\end{equation}
and $\tilde Q_1(0) = Q_1(0)$. 

Denote 
\begin{equation} \label{defr}
 	\hat M(\rho) := M(\rho) - \tilde M(\rho), \quad r(x, \rho, \theta) := \hat M(\theta) D(x, \rho, \theta), 
 	\quad \tilde r(x, \rho, \theta) := \hat M(\theta) \tilde D(x, \rho, \theta).
\end{equation}

Let $\ga_0$ be a bounded closed contour in the $\rho$--plane, oriented counterclockwise and enclosing
the sets $\Lambda_{\pm}$, $\tilde \Lambda_{\pm}$ and $\{ 0 \}$, and let $\ga_{\pm}$ be a two--sided cut 
along the ray $\{ \rho \colon \pm \rho > 0, \rho \notin \mbox{int}\, \ga_0  \}$. Denote $\ga = \ga_{-} \cup \ga_0 \cup \ga_{+}$
and $J_{\ga} = \{ \rho \colon \rho \notin \ga \cup \mbox{int}\, \ga_0\}$.

\begin{lem} \label{lem:contour}
The following relations hold:
\begin{equation} \label{cont1}
 	\Omega(x) \tilde \vv(x, \rho) = \vv(x, \rho) + \frac{1}{2 \pi i} \int_{\gamma} \vv(x, \theta) \tilde r(x, \rho, \theta) \, d \theta,
\end{equation}
\begin{equation} \label{cont2}
 	r(x, \rho, \theta) - \tilde r(x, \rho, \theta) + \frac{1}{2 \pi i} \int_{\gamma} r(x, \xi, \theta) \tilde r(x, \rho, \xi) \, d \xi =
 	\hat M(\theta) \vv^*(x, \theta) \Lambda(x) \tilde \vv(x, \rho).
\end{equation}
\end{lem}

\begin{proof}
1. Consider the contour $\ga_R := \ga \cap \{ \rho \colon |\rho| \le R \}$  
oriented counterclockwise and the contour  
$\ga_R^0 := \ga_R$ oriented clockwise. By Cauchy's integral formula, 
$$
  \mathcal{P}_{1k}(x, \rho) - \Omega(x) \delta_{1k} = \frac{1}{2 \pi i} 
  \int_{\ga_R^0} \frac{\mathcal{P}_{1k}(x, \theta) - \Omega(x) \delta_{1k}}{\rho - \theta} \, d \theta, \quad 
  k = 1, 2, \quad x \ge 0, \quad \rho \in \mbox{int}\, \ga_R^0, 
$$
where $\de_{jk}$ is the Kroneker delta.
Using \eqref{eqhs} and Lemma~\ref{lem:asymptP}, we obtain
$$
	\lim_{R \to \infty} \int_{|\theta| = R} \frac{\mathcal{P}_{1k}(x, \theta) - \Omega(x) \delta_{1k}}{\rho - \theta} \, d \theta = 0.
$$
Consequently,
\begin{equation} \label{smeq3}
 	\mathcal{P}_{1k}(x, \rho) = \Omega(x) \delta_{1k} + \frac{1}{2 \pi i} \int_{\ga} \frac{\mathcal{P}_{1k}(x, \theta)}
 	{\rho - \theta} \, d \theta, \quad k = 1, 2, \quad x \ge 0, \quad \rho \in J_{\ga}.
\end{equation}
Here and below the integral is understood in sense of a principal value: $\int_{\ga} := \lim_{R \to \iy} \int_{\ga_R}$.

According to \eqref{defP},
\begin{equation} \label{smeq4}
  \vv(x, \rho) = \mathcal{P}_{11}(x, \rho) \tilde \vv(x, \rho) + \mathcal{P}_{12}(x, \rho) \tilde \vv'(x, \rho).
\end{equation}
Substituting \eqref{smeq3} into \eqref{smeq4} and using \eqref{Pj12}, we obtain
\begin{multline*}
  \vv(x, \rho) = \Omega(x) \tilde \vv(x, \rho) + \frac{1}{2 \pi i} \int_{\ga} \bigl[
  (\vv(x, \theta) \tilde {\Phi^*}'(x, \theta) - \Phi(x, \theta) \tilde {\vv^*}'(x, \theta)) \tilde \vv(x, \rho) + \\+
  (\Phi(x, \theta) \tilde \vv^*(x, \theta) - \vv(x, \theta) \tilde \Phi^*(x, \theta)) \tilde \vv'(x, \rho) \bigr]
  \frac{d \, \theta}{\rho - \theta}.
\end{multline*}                                                                
In view of \eqref{Phiexp}, \eqref{Phiexp*} and \eqref{defD}, this yields
$$
	\vv(x, \rho) = \Omega(x) \tilde \vv(x, \rho) - \frac{1}{2 \pi i} \int_{\ga} \vv(x, \theta) \hat M(\theta) \tilde D(x, \rho, \theta) \, d \theta,
$$
since the terms with $S(x, \theta)$ and $\tilde S^*(x, \theta)$ vanish by Cauchy's theorem.
Taking \eqref{defr} into account, we arrive at \eqref{cont1}. The relation \eqref{cont1} holds for
all $x \ge 0$ and $\rho \in \mathbb{C}$, since both left and right hand sides are entire functions of $\rho$. 

2. Since 
$$
\frac{1}{\rho - \theta}\left(\frac{1}{\rho - \xi} - \frac{1}{\theta - \xi}\right) = \frac{1}{(\rho - \xi)(\xi - \theta)},
$$
analogously to \eqref{smeq3}, for $x \ge 0$ and $\rho, \theta \in J_{\ga}$, we derive
$$
 \frac{\mathcal{P}_{21}(x, \rho) - \mathcal{P}_{21}(x, \theta)}{\rho - \theta} = -\Lambda(x) + \frac{1}{2 \pi i} \int_{\ga}
 \frac{\mathcal{P}_{21}(x, \xi)}{(\rho - \xi)(\xi - \theta)} \, d \xi,
$$
$$
  \frac{\mathcal{P}_{jk}(x, \rho) - \mathcal{P}_{jk}(x, \theta)}{\rho - \theta} = \frac{1}{2 \pi i} \int_{\ga}
  \frac{\mathcal{P}_{jk}(x, \xi)}{(\rho - \xi)(\xi - \theta)} \, d \xi, \quad (j, k) \ne (2, 1).
$$
For an arbitrary vector function $Y = Y(x)$, in view of \eqref{inverseform} and \eqref{defP}, we have
$$
	\mathcal{P} \left[ \begin{array}{l} Y \\ Y' \end{array} \right] = 
	\left[ \begin{array}{l} \vv \\ \vv' \end{array} \right] \langle \tilde \Phi^*, Y \rangle -
	\left[ \begin{array}{l} \Phi \\ \Phi' \end{array} \right] \langle \tilde \vv^*, Y \rangle.
$$
Therefore
\begin{multline*}
 	\frac{\mathcal{P}(x, \rho) - \mathcal{P}(x, \theta)}{\rho - \theta} 
 	\left[ \begin{array}{l} Y(x) \\ Y'(x) \end{array} \right] = 
 	\left[ \begin{array}{l} 0 \\ -\Lambda(x) Y(x) \end{array} \right] + \frac{1}{2 \pi i} \int_{\ga}
 	\Biggl\{ \left[ \begin{array}{l} \vv(x, \xi) \\ \vv'(x, \xi) \end{array} \right]
 	\langle \tilde \Phi^*(x, \xi), Y(x) \rangle - \\ 
	\left[ \begin{array}{l} \Phi(x, \xi) \\ \Phi'(x, \xi) \end{array} \right]
 	\langle \tilde \vv^*(x, \xi), Y(x) \rangle
 	\Biggr\}\frac{d \xi}{(\rho - \xi)(\xi - \theta)}, 
\end{multline*}
\begin{multline} \label{sum1}
    [{\vv^*}'(x, \theta), -\vv^*(x, \theta)]\frac{\mathcal{P}(x, \rho) - \mathcal{P}(x, \theta)}{\rho - \theta}  
    \left[ \begin{array}{l} \tilde \vv(x, \rho) \\ \tilde \vv'(x, \rho) \end{array} \right] = 
    \vv^*(x, \theta) \Lambda(x) \tilde \vv(x, \rho) + \\
	\frac{1}{2 \pi i} \int_{\ga}
    \frac{(\langle \vv^*(x, \theta), \vv(x, \xi) \rangle \langle \tilde \Phi^*(x, \xi), \tilde \vv(x, \rho) \rangle -
    \langle \vv^*(x, \theta), \Phi(x, \xi) \rangle \langle \tilde \vv^*(x, \xi), \tilde \vv(x, \rho) \rangle ) d \xi}{(\rho - \xi)(\xi - \theta)}.
\end{multline}
Using \eqref{defP}, we obtain
\begin{equation} \label{sum2}
  [{\vv^*}'(x, \theta), -\vv^*(x, \theta)] \mathcal{P}(x, \rho) 
  \left[ \begin{array}{l} \tilde \vv(x, \rho) \\ \tilde \vv'(x, \rho) \end{array} \right] =
  \langle \vv^*(x, \theta), \vv(x, \rho) \rangle.
\end{equation}
By virtue of \eqref{inverseform} and \eqref{defP}, we have
$$
	\left[\begin{array}{ll}
	{\Phi^*}' & -\Phi^* \\ -{\vv^*}' & \vv^* 
	\end{array}\right] \mathcal{P} =
	\left[\begin{array}{ll}
	\tilde {\Phi^*}' & -\tilde \Phi^* \\ -\tilde {\vv^*}' & \tilde \vv^* 
	\end{array}\right]. 
$$
Hence
\begin{equation} \label{sum3}
  [{\vv^*}'(x, \theta), -\vv^*(x, \theta)] \mathcal{P}(x, \theta) 
  \left[ \begin{array}{l} \tilde \vv(x, \rho) \\ \tilde \vv'(x, \rho) \end{array} \right] =
  \langle \tilde \vv^*(x, \theta), \tilde \vv(x, \rho) \rangle.
\end{equation}

Combining \eqref{sum1}, \eqref{sum2} and \eqref{sum3} and using \eqref{defD}, we obtain
\begin{multline*}
 	D(x, \rho, \theta) - \tilde D(x, \rho, \theta) =
 	\vv^*(x, \theta) \Lambda(x) \tilde \vv(x, \rho) + \frac{1}{2 \pi i} \int_{\ga} (D(x, \xi, \theta) \tilde M^*(\xi) 
 	\tilde D(x, \rho, \xi) - \\ D(x, \xi, \theta) M(\xi) \tilde D(x, \rho, \xi)) \, d \xi.
\end{multline*}
Using \eqref{defr}, we arrive at \eqref{cont2}.
\end{proof}

Denote
$$
 	\tilde \Omega(x) := \frac{1}{2} \left( \tilde P_{+}(x) P_{-}^*(x) + \tilde P_{-}(x) P_{+}^*(x) \right),
$$
$$
 	\tilde \Lambda(x) := \frac{1}{2 i} \left( \tilde P_{-}(x) P_{+}^*(x) - \tilde P_{+}(x) P_{-}^*(x) \right)
$$
(compare with \eqref{defOmega}, \eqref{defLambda}).
Symmetrically with the relations of Lemma~\ref{lem:contour}, we have
\begin{equation} \label{cont3}
 	\tilde \Omega(x) \vv(x, \rho) = \tilde \vv(x, \rho) - \frac{1}{2 \pi i} \int_{\ga} \vv(x, \theta) r(x, \rho, \theta) \, d \theta,
\end{equation}
\begin{equation} \label{cont4}
 	r(x, \rho, \theta) - \tilde r(x, \rho, \theta) + \frac{1}{2 \pi i} \int_{\gamma} \tilde r(x, \xi, \theta) r(x, \rho, \xi) \, d \xi =
 	-\hat M(\theta) \tilde \vv^*(x, \theta) \tilde \Lambda(x) \vv(x, \rho).
\end{equation}

Consider the Banach space $C(\ga; \mathbb{C}^m)$ of row-vectors $Y(\rho) = [y_k(\rho)]^T_{k = \overline{1, m}}$
with components continuous on $\ga$ with the norm
$$
 	\| Y(\rho) \| = \max_{\rho \in \ga} \max_{k = \overline{1, m}} |y_k(\rho)|. 
$$
For each fixed $x \ge 0$, we define the linear operators $A$ and $\tilde A$ in $C(\ga; \mathbb{C}^m)$:
$$
 	Y(\rho) A = Y(\rho) - \frac{1}{2 \pi i} \int_{\ga} Y(\theta) r(x, \rho, \theta) \, d \theta, \quad \rho \in \ga,
$$
$$
 	Y(\rho) \tilde A = Y(\rho) + \frac{1}{2 \pi i} \int_{\ga} Y(\theta) \tilde r(x, \rho, \theta) \, d \theta, \quad \rho \in \ga,
$$
Since by definition these operators belong to the ring for which the elements of $C(\ga; \mathbb{C}^m)$ form a right
module, we write operators to the right of operands.

Using \eqref{defD}, \eqref{estD}, \eqref{intM} and \eqref{defr}, one can easily obtain the following estimates
$$
 	\int_{\ga} \| r(x, \rho, \theta) \| |d \theta| \le C_x, \quad \rho \in \ga,
$$
$$
 	\int_{\ga} \| r(x, \rho_1, \theta) - r(x, \rho_2, \theta) \| |d \theta| \le C_x |\rho_1 - \rho_2|, \quad
 	\rho_1, \, \rho_2 \in \ga,
$$
and similar estimates for $\tilde r(x, \rho, \theta)$. Consequently,
for each fixed $x \ge 0$, the operators $A$ and $\tilde A$ are bounded and compact in $C(\ga; \mathbb{C}^m)$.

\begin{lem} For each fixed $x \ge 0$, the following relations hold.
\begin{equation} \label{Aphi}
 	\vv(x, \rho) \tilde A = \Omega(x) \vv(x, \rho), \quad \tilde \vv(x, \rho) A = \tilde \Omega(x) \vv(x, \rho),
\end{equation}
\begin{equation} \label{tAA}
 	Y(\rho) \tilde A A = Y(\rho) + \frac{1}{2 \pi i} \int_{\ga} Y(\theta) \hat M(\theta) \tilde \vv^*(x, \theta) \, d \theta \cdot
 	\tilde \Lambda(x) \vv(x, \rho).
\end{equation}
\end{lem}

\begin{proof}
The relations \eqref{Aphi} follow immediately from \eqref{cont1} and \eqref{cont3}. 
Since 
$$
Y(\rho) A \tilde A = Y(\rho) + \frac{1}{2 \pi i} \int_{\ga} Y(\theta) \left( \tilde r(x, \rho, \theta) - r(x, \rho, \theta) - 
\int_{\ga} \tilde r(x, \xi, \theta) r(x, \rho, \xi) \, d \xi\right) \, d \theta,
$$
the relation \eqref{tAA} follows from \eqref{cont4}.
\end{proof}

\medskip

{\it Remark.} For the matrix Sturm-Liouville operator (when $Q_1(x) \equiv 0$),
we have $\Omega(x) = I$, $\Lambda(x) = 0$. Therefore, the operators $A$ and $\tilde A$ are inverses of each other. 

\medskip

\begin{thm} \label{thm:homo}
For each fixed $x \ge 0$, the subspace of solutions of the homogeneous equation 
\begin{equation} \label{homo}
Y(\rho) \tilde A = 0
\end{equation}
in the Banach space $C(\ga; \mathbb{C}^m)$ has dimension $m - \mbox{rank}\, \Omega(x)$.
In the particular case $\det \Omega(x) \ne 0$, equation~\eqref{homo} has only the trivial solution. 
\end{thm}

\begin{proof}
Let $Y(\rho)$ be a solution of \eqref{homo}. Then by \eqref{tAA} 
$$
 	Y(\rho) \tilde A A = Y(\rho) - K(x) \vv(x, \rho) = 0,
$$
where $K(x)$ is a row vector.
Using \eqref{Aphi}, we get
$$
 	Y(\rho) \tilde A = K(x) \vv(x, \rho) \tilde A = K(x) \Omega(x) \tilde \vv(x, \rho) = 0.
$$
Let $\mathcal{K}(x)$ be the space of solutions of this equation for each fixed $x \ge 0$.
It is clear that $\dim \mathcal{K}(x) = m - \mbox{rank} \, \Omega(x)$, and $Y(\rho)$ is
a solution of \eqref{homo} if and only if $Y(\rho) = K(x) \vv(x, \rho)$, $K(x) \in \mathcal{K}(x)$,
this yields the assertion of the lemma.
\end{proof}

We assume that $\det \Omega(x) \ne 0$. For example, this holds when the matrices $Q_1(x)$
and $\tilde Q_1(x)$ are Hermitian ($Q_1(x) = - Q_1^{\dagger}(x)$, where the sign $\dagger$
denotes the conjugate transposition). 

For each fixed $x \ge 0$ we consider in $C(\ga, \mathbb{C}^{m \times m})$
the linear equation
\begin{equation} \label{main}
 	\tilde \vv(x, \rho) = z(x, \rho) + \frac{1}{2 \pi i} \int_{\ga} z(x, \theta) \tilde r(x, \rho, \theta) \, d \theta
\end{equation}
with respect to $z(x, \rho)$. Equation \eqref{main} is called {\it the main equation} of Inverse Problem 1.
Eq. \eqref{Aphi} and Theorem~\ref{thm:homo} yield

\begin{thm} \label{thm:main}
For each fixed $x \ge 0$, such that $\det \Omega(x) \ne 0$,
equation \eqref{main} has a unique solution $z(x, \rho) = (\Omega(x))^{-1} \vv(x, \rho)$. 	
\end{thm}

\bigskip

{\large \bf 5. Constructive solution for Inverse Problem 1}

\bigskip

In this section, using the main equation \eqref{main} and Theorem~\ref{thm:main},
we provide a constructive algorithm for the solution of Inverse Problem~1.

In a similar way to Lemma~\ref{lem:contour}, it can be proved that for $\rho \in J_{\ga}$,
$$
 	\Omega(x) \tilde \Phi(x, \rho) = \Phi(x, \rho) + \frac{1}{2 \pi i} \int_{\ga} \vv(x, \rho) \hat M(\theta)
 	\frac{\langle \tilde \vv^*(x, \theta), \tilde \Phi(x, \rho) \rangle}{\rho - \theta} \, d \theta.
$$
Therefore, if $\det \Omega(x) \ne 0$ and $z(x, \rho)$ is already known, we can build the matrix function 
$w(x, \rho) := (\Omega(x))^{-1} \Phi(x, \rho)$ by the formula
\begin{equation} \label{defw}
 	w(x, \rho) = \tilde \Phi(x, \rho) - \frac{1}{2 \pi i} \int_{\ga} z(x, \theta) \bar M(\theta) 
 	\frac{\langle \tilde \vv^*(x, \theta), \tilde \Phi(x, \rho) \rangle}{\rho -\theta} \, d \theta.
\end{equation}

Analogously, we construct the main equation of the inverse problem for the pencil $L^* = (\ell^*_{\rho}, U^*_{\rho})$
(see definitions \eqref{lU*}):
\begin{equation} \label{main*}
 	\vv^*(x, \rho) = z^*(x, \rho) + \frac{1}{2 \pi i} \int_{\ga} \tilde D^*(x, \rho, \theta) \hat M(\theta) z^*(x, \theta) \, d \theta.
\end{equation}
Here we have used the relation $\hat M^*(\rho) \equiv \hat M(\rho)$, following from \eqref{eqMM*}.
Denote
$$
 	\Omega^*(x) := \frac{1}{2} \left( \tilde P_{+}(x) P_{-}^*(x) + \tilde P_{-}(x) P_{+}^*(x) \right) = \tilde \Omega(x).
$$
Suppose $\det \Omega^*(x) \ne 0$ for all $x \ge 0$. 
Then eq. \eqref{main*} has the unique solution $z^*(x, \rho) = \vv^*(x, \rho) (\Omega^*(x))^{-1}$,
and we also construct
\begin{equation} \label{defw*}
 	w^*(x, \rho) = \Phi^*(x, \rho) (\Omega^*(x))^{-1} = \tilde \Phi^*(x, \rho) - \frac{1}{2 \pi i} \int_{\ga}
 	\frac{\langle \tilde \Phi^*(x, \rho), \tilde \vv(x, \theta)\rangle}{\rho - \theta} \hat M(\theta) z^*(x, \theta) \, d \theta.
\end{equation}

Using \eqref{inverseform}, we get
$$
 	0 = \vv \Phi^* - \Phi \vv^* = \Omega (z w^* - w z^*) \Omega^*.
$$
$$
 	I = \vv {\Phi^*}' - \Phi {\vv^*}' = \Omega (z w^* - w z^*) {\Omega^*}' + \Omega (z {w^*}' - w {z^*}') \Omega^*.
$$
Therefore, 
$$
 	z w^* - w z^* = 0, 		
$$
\begin{equation} \label{Omega2}
 	(z {w^*}' - w {z^*}')^{-1} = \Omega^* \Omega. 	
\end{equation}

It follows from \eqref{wronphi} that
$$
 	\langle z^* \Omega^*, \Omega z \rangle = 0.
$$
Consequently,
\begin{equation} \label{Omega3}
  	z^* ({\Omega^*}' \Omega - \Omega^* \Omega') z = z^* \Omega^* \Omega z' - {z^*}' \Omega^* \Omega z.
\end{equation}
If the functions $z$, $z^*$, $w$ and $w^*$ are known, one can find $\Omega^* \Omega$ from \eqref{Omega2}.
Then the right-hand side of \eqref{Omega3} is known, and we can calculate ${\Omega^*}' \Omega - \Omega^* \Omega'$.
Differentiating $\Omega^* \Omega$, we find ${\Omega^*}' \Omega + \Omega^* \Omega'$. Then we can find
$\Omega^* \Omega'$ and calculate 
$$
	a(x) := \Omega^{-1} \Omega' = (\Omega^* \Omega)^{-1} \Omega^* \Omega'.
$$
Solving the Cauchy problem
$$
 	 \Omega'(x) = a(x) \Omega(x), \quad \Omega(0) = I,
$$
we get $\Omega(x)$ and then can find $\vv(x, \rho)$.

In the particular case when $Q_1(x)$ and $\tilde Q_1(x)$ are skew--Hermitian, we see from \eqref{cauchyP} and 
\eqref{cauchyP*} that $P_{+}^* = P_{-}^{\dagger}$ and $P_{-}^* = P_{+}^{\dagger}$. Hence $\Omega^* = \Omega^{\dagger}$,
and we can find $\Omega$ as a matrix square root from $\Omega^* \Omega$.

We arrive at the following algorithm for the solution of Inverse Problem 1.

\medskip

{\bf Algorithm 1.} Let the Weyl function of the pencil $L$ be given.
\begin{enumerate}
\item Find $h_s$, $s = 0, 1$, and $Q_1(0)$ from the asymptotics in Lemma~\ref{lem:asymptM}.
\item Choose the model pencil $\tilde L$ satisfying \eqref{intM}.
\item Construct the matrix functions $\tilde \vv(x, \rho)$,
$\tilde \Phi(x, \rho)$, $\tilde M(\rho)$, $\tilde r(x, \rho, \theta)$ and the similar functions for $\tilde L^*$. 
\item Find $z(x, \rho)$ and $z^*(x, \rho)$, solving the main equations \eqref{main} and \eqref{main*}.
\item Construct $w(x, \rho)$ and $w^*(x, \rho)$ via \eqref{defw} and \eqref{defw*}, then get $\Omega^*(x) \Omega(x)$ 
via \eqref{Omega2}.
\item Find $\Omega(x)$ and $\vv(x, \rho) = \Omega(x) z(x, \rho)$.	
\item Substituting $\vv(x, \rho)$ into \eqref{eqv}, obtain
$Q_s(x)$, $x \ge 0$.
\end{enumerate}

Note that Algorithm 1 works only for the case $\det \Omega(x) \ne 0$ and $\det \Omega^*(x) \ne 0$
for all $x \ge 0$. 
Further we provide a modification of this algorithm for the general case.
We will write that $\tilde Q_1$ belongs to the class $\mathcal{Q}(\de)$, $\de \ge 0$,
if the following conditions hold
\begin{enumerate}
\item $\tilde Q_1(x) = Q_1(x)$, $x \in [0, \de]$;
\item $\| \tilde Q_1(x) \| \le \| Q_1(\de) \| $, $x \in (\de, \iy)$;
\item $\int_{\de}^{\iy} \| \tilde Q_1(x) \| \, dx < \| Q_1(\de) \|$.
\end{enumerate}

{\bf Algorithm 2.} Let the Weyl function of the pencil $L$ be given.
\begin{enumerate}
\item Find $h_s$, $s = 0, 1$, and $Q_1(0)$ from the asymptotics in Lemma~\ref{lem:asymptM}.
\item Put $k := 1$, $\de_0 := 0$.                                     
\item We assume that for $x \in [0, \de_{k - 1}]$ the potential $Q_1(x)$
is already known. Choose the model pencil $\tilde L$ satisfying \eqref{intM}
and with $\tilde Q_1$ belonging to $\mathcal{Q}(\de_{k - 1})$.
\item Step 3 of Algorithm 1.
\item Choose the largest $\de_k$, $\de_{k - 1} < \de_k \le \iy$, such that the main equations \eqref{main} and \eqref{main*} have unique solutions
for $x \in (\de_{k-1}, \de_{k})$, find these solutions.
\item Implement steps 5--7 of Algorithm 1 for $x \in (\de_{k - 1}, \de_k)$.
\item If $\de_k = \iy$, then terminate the algorithm, otherwise put $k := k + 1$ and go to step 3.             
\end{enumerate}

\begin{thm}
Algorithm 2 terminates after a finite number of steps.
\end{thm}

\begin{proof}
1. Consider the pencil $\tilde L$ such that $\tilde Q_1 \in \mathcal{Q}(\de)$, $\de \ge 0$.

We have $P_{\pm}(x) = \tilde P_{\pm}(x)$ and $P^*_{\pm}(x) = \tilde P^*_{\pm}(x)$ for $x \in [0, \de]$.
For $x > \de$, the following integral representations are valid
$$
 	P_{\pm}(x) = P_{\pm}(\de) \pm \int_{\de}^{x} Q_1(t) P_{\pm}(t) \, dt,
$$  
$$
 	\tilde P_{\pm}^*(x) = \tilde P_{\pm}^*(\de) \pm \int_{\de}^x \tilde P_{\pm}^*(t) \tilde Q_1(t) \, dt.
$$
Since $\tilde P_{\pm}^*(\de) = P_{\pm}^*(\de) = (P_{\mp}(\de))^{-1}$, using \eqref{defOmega}, we obtain
$$
 	\Omega(x) = I + B_{\de}(x) = I + \frac{1}{2}(B_{+}(x) + B_{-}(x)), 
$$
\begin{multline*}
 	B_{\pm}(x) := \pm \int_{\de}^x Q_1(t) P_{\pm}(t) \, dt \tilde P^*_{\mp}(\de) 
 	\mp P_{\pm}(\de) \int_{\de}^x \tilde P^*_{\mp}(t) \tilde Q_1(t) \, dt \\ -
 	\int_{\de}^x Q_1(t) P_{\pm}(t) \, dt \cdot \int_{\de}^x \tilde P^*_{\mp}(t) \tilde Q_1(t) \, dt.
\end{multline*}

Let us denote by $K(L)$ different constants that depend on the pencil $L$ (namely, on the unknown potential $Q_1$)
and do not depend on $\de \ge 0$ and on the choice of $\tilde Q_1 \in \mathcal{Q}(\de)$.
Then we have 
$$
	\int_0^{\iy} \| Q_1(x) \| \, dx < K(L), \quad \max_{x \ge 0} \| Q_1(x) \| < K(L).
$$
For every $\de \ge 0$ and $\tilde Q_1 \in \mathcal{Q}_{\de}$,
$$
	\int_0^{\iy} \| \tilde Q_1(x) \| \, dx < K(L), \quad \max_{x \ge 0} \| \tilde Q_1(x) \| < K(L).
$$
By virtue of Lemma~\ref{lem:P},
$$
 	\| P_{\pm}(x) \|, \, \| P^*_{\pm}(x) \| < K(L).
$$
Therefore, one can easily obtain two estimates
\begin{equation} \label{estB1}
 	\| B_{\de}(x) \| \le K(L) \left[ \| Q_1(\de) \| + \int_{\de}^x \| Q_1(t) \| \, dt + 
 	 \| Q_1(\de) \|\int_{\de}^x \| Q_1(t) \| \, dt \right],
\end{equation}
\begin{equation} \label{estB2}
   \| B_{\de}(x) \| \le K(L) (x - \de).
\end{equation}

2. Let us construct a partition $0 = x_0 < x_1 < x_2 < \dots < x_{s - 1} < x_{s} = \infty$
such that for each $k$ from $1$ to $s$, if we put $\de = x_{k-1}$ in step 1 of the proof and construct $\tilde Q_1(x)$,
we get $\| B_{\de_k}(x) \| < 1/2$, $x \in (x_{k-1}, x_k)$. Hence $\det \Omega(x) \ne 0$ for $x \in (x_{k - 1}, x_k)$.
For $\Omega^*(x)$ the proof is similar. Hence
the main equations \eqref{main} and \eqref{main*} are uniquely solvable in this interval.

First, it follows from \eqref{estB1}, that
$$
 	\lim_{\de \to \infty} \max_{x \ge \de} \| B_{\de}(x) \| = 0.
$$
Therefore we can choose $x_*$ such that for every $x$ in $(x_{*}, \iy)$ we have $\| B_{x_*}(x) \| < 1/2$.

Further, we divide the segment $[0, x_*]$ into sufficiently small segments $[x_{k - 1}, x_k]$, $k = \overline{1, s-1}$, 
such that $\| B_{x_{k-1}}(x) \| < 1/2$ for $x \in (x_{k - 1}, x_k)$
by virtue of \eqref{estB2}. 

3. The partition $\{ x_k \}$, constructed at the previous step, is independent of the particular choice of $\tilde Q_1$,
but it depends on the pencil $L$ which is not known a priori. Therefore the partition $\{ x_k \}$ is unknown,
and we can not directly use it in Algorithm~2. However, we can establish 
connection between $\{ x_k \}$ and $ \{ \de_k \}$.
It is easy to show by induction, that $\de_k \ge x_k$ for all $k = \overline{1, s}$. Consequently, 
the number of steps in Algortihm~2 is finite.

\end{proof}

\medskip

{\large \bf Appendix}

\medskip

Here we provide the proof of Theorem~\ref{thm:FSS}. By {\it Step 1}, we reduce the differential 
equation $\ell_{\rho}(Y) = 0$ to an integral one.
Then we confine ourselves to the case $\rho \in \Pi_+$. 
We construct the Jost-type solution $E_{+}(x, \rho)$ by {\it Step 2} and the Birkhoff-type solution
$E_{-}(x, \rho)$ by {\it Step 3}. Some parts of {\it Step 3} are analoguous to {\it Step 2}, so 
they are described in short. By {\it Step 4}, we prove that the constructed system
of solutions $\{ E_{+}(x, \rho), E_{-}(x, \rho) \}$
is linearly independed for each fixed $\rho \ne 0$. The case $\rho \in \Pi_{-}$ can be considered in a similar way.
In this case, $E_{-}(x, \rho)$ is the Jost-type solution and $E_{+}(x, \rho)$ is Birkhoff-type. 

\medskip

\begin{proof}[Proof of Theorem~\ref{thm:FSS}]

{\bf Step 1.} 
Consider the matrix functions $E_{\pm}^0(x, \rho) := \exp(\pm i \rho x) P_{\mp}(x)$, where
$P_{\pm}(x)$ are defined by \eqref{cauchyP}. It is easy to check, that $E^0_{\pm}(x)$ 
form a fundamental system of solutions for the differential equation
$$
Y'' + Q'_1(x) (i \rho I - Q_1(x))^{-1} Y' + (\rho^2 I + 2 i \rho Q_1(x) - Q_1^2(x)) Y = 0.
$$
Rewrite equation $\ell_{\rho}(Y) = 0$ in the form
\begin{gather} \label{eqvrewr}
Y'' + Q'_1(x) (i \rho I - Q_1(x))^{-1} Y' + (\rho^2 I + 2 i \rho Q_1(x) - Q_1^2(x)) Y = F(x, \rho, Y), \\
\label{defF}
F(x, \rho, Y) := Q'_1(x) (i \rho I - Q_1(x))^{-1} Y' - (Q_1^2(x) + Q_0(x))Y.
\end{gather}
Apply the method of variation of parameters to this equation. Every solution of \eqref{eqvrewr} can be represented in 
the form 
$$
	Y(x, \rho) = E_{+}^0(x, \rho) A(x, \rho) + E_{-}^0(x, \rho) B(x, \rho),
$$
where coefficient matrices $A(x, \rho)$ and $B(x, \rho)$ satisfy the system
$$
 	\left[\begin{array}{cc} E_+^0(x, \rho) & E_-^0 (x, \rho) \\ {E_+^0}'(x, \rho) & {E_-^0}'(x, \rho) \end{array}\right] \cdot
 	\left[\begin{array}{c} A'(x, \rho) \\ B'(x, \rho) \end{array}\right] = 
 	\left[\begin{array}{c} 0 \\ F(x, \rho, Y) \end{array}\right].
$$
Using Lemma~\ref{lem:P}, we find the inverse matrix
$$
 	\left[\begin{array}{cc} E_+^0(x, \rho) & E_-^0 (x, \rho) \\ {E_+^0}'(x, \rho) & {E_-^0}'(x, \rho) \end{array}\right]^{-1} =
 	\frac{1}{2} \left[\begin{array}{cc} \exp(-i \rho x) P_{+}^*(x) & \exp(-i \rho x) P_{+}^*(x) (i \rho I - Q_1(x))^{-1} \\
 	                                    \exp(i \rho x) P_{-}^*(x) & -\exp(i \rho x) P_{-}^*(x) (i \rho I - Q_1(x))^{-1}
 	                  \end{array}\right].                       	
$$
Consequently,
$$
 	A'(x, \rho) = \frac{1}{2} \exp(-i \rho x) P_{+}^*(x)(i \rho I - Q_1(x))^{-1} F(x, \rho, Y),
$$
$$
 	B'(x, \rho) = -\frac{1}{2} \exp(i \rho x) P_{-}^*(x) (i \rho I - Q_1(x))^{-1} F(x, \rho, Y),
$$
\begin{multline} \label{inteq}
Y(x, \rho) = \exp(i \rho x) P_{-}(x) A(a, \rho) + \exp(-i \rho x) P_{+}(x) B(a, \rho) + \frac{1}{2}
\int_a^x \bigl\{ \exp(i \rho (x - t)) P_{-}(x) P_{+}^*(t) \\ - \exp(i \rho (t - x)) P_{+}^*(x) P_{-}^*(t)\bigr\} (i \rho I - Q_1(t))^{-1}
F(t, \rho, Y) \, dt.
\end{multline}

{\bf Step 2.}
Let $E_{+}(x, \rho)$ be the solution of the equation
\begin{multline} \label{inteqJost}
  E_{+}(x, \rho) = \exp(i \rho x) P_{-}(x) - \frac{1}{2} \int_x^\iy 
  (\exp(i \rho (x - t)) P_{-}(x) P_{+}^*(t) \\ - \exp(i \rho (t - x)) P_{+}^*(x) P_{-}^*(t)) (i \rho I - Q_1(t))^{-1}
F(t, \rho, E_{+}) \, dt.
\end{multline} 
Since this is a particular case of \eqref{inteq}, the function $E_{+}(x, \rho)$ satisfy $\ell_{\rho}(E_{+}) = 0$.
We transform \eqref{inteqJost} by means of the replacement $E_{+}(x, \rho) = \exp(i \rho x) Z(x, \rho)$ to the equation
\begin{multline} \label{inteqZ}
   Z(x, \rho) = P_{-}(x) - \frac{1}{2}\int_x^\iy \bigl[P_{-}(x) P_{+}^*(t) - \exp(2 i \rho(t - x)) P_{+}(x) P_{-}^*(t)\bigr]
   (i \rho I - Q_1(t))^{-1} \\ \bigl[Q_1'(t) (i \rho I - Q_1(t))^{-1} (Z'(t, \rho) + i \rho Z(t, \rho)) - 
   (Q_1^2(t) + Q_0(t)) Z(t, \rho)\bigr] \, dt.
\end{multline}
The method of successive approximations gives
\begin{equation} \label{Z0}
 	Z_0(x, \rho) = P_{-}(x), \quad Z_0'(x) = -Q_1(x) P_{-}(x),
\end{equation}
$$
 	Z_{k + 1}(x, \rho) = \int_x^{\iy} F_1(x, t, \rho) Z_k(t, \rho) \, dt + \int_x^{\iy} F_2(x, t, \rho) Z'_k(t, \rho) \, dt,
$$
$$
 	Z'_{k + 1}(x, \rho) = \int_x^{\iy} F_3(x, t, \rho) Z_k(t, \rho) \, dt + \int_x^{\iy} F_4(x, t, \rho) Z'_k(t, \rho) \, dt, 
$$
\begin{equation} \label{seriesZ}
    Z(x, \rho) = \sum_{k = 0}^{\iy} Z_k(x, \rho),
\end{equation}
where
\begin{gather*}
 	F_1(x, t, \rho) = -\frac{1}{2}\bigl[P_{-}(x) P_{+}^*(t) - \exp(2 i \rho(t - x)) P_{+}(x) P_{-}^*(t)\bigr] 
   (i \rho I - Q_1(t))^{-1} \\ \cdot \bigl[Q_1'(t) (i \rho I - Q_1(t))^{-1} i \rho - Q_1^2(t) - Q_0(t) \bigr], \\ 
 	F_2(x, t, \rho) = -\frac{1}{2}\bigl[P_{-}(x) P_{+}^*(t) - \exp(2 i \rho(t - x)) P_{+}(x) P_{-}^*(t)\bigr] 
   (i \rho I - Q_1(t))^{-1} Q_1'(t) (i \rho I - Q_1(t))^{-1}, \\
 	F_3(x, t, \rho) = \frac{1}{2} \bigl[Q_1(x) (P_{-}(x) P_{+}^*(t) + \exp(2 i \rho (t - x)) P_{+}(x) P_{-}^*(t)) 
 	- 2 i \rho \exp(2 i \rho (t - x))P_{+}(x) P^*_{-}(t)\bigr] \\ \cdot (i \rho I - Q_1(t))^{-1} \bigl[Q_1'(t) (i \rho I - Q_1(t))^{-1} i \rho - Q_1^2(t) - Q_0(t) \bigr],\\
 	F_4(x, t,\rho) = \frac{1}{2} \bigl[Q_1(x) (P_{-}(x) P_{+}^*(t) + \exp(2 i \rho (t - x)) P_{+}(x) P_{-}^*(t)) 
 	- 2 i \rho \exp(2 i \rho (t - x))P_{+}(x) P^*_{-}(t)\bigr] \\ \cdot (i \rho I - Q_1(t))^{-1} 
 	(i \rho I - Q_1(t))^{-1} Q_1'(t) (i \rho I - Q_1(t))^{-1}.
\end{gather*}

Since $Q_1(t)$ is bounded, for each $\rho_* > 0$ there exists $x_* \ge 0$ such that 
\begin{equation} \label{Q*}
	\| Q_1(t) \| \le \rho_* / 2, \quad t \ge x_*.
\end{equation}
Consequently, 
$$
	\| (i \rho I - Q_1(t))^{-1} \| \le \frac{2}{|\rho|}, \quad |\rho| \ge |\rho^*|.
$$
We introduce auxiliary functions
$$
 	\mathcal{F}(x) := \exp\left( \int_0^x \| Q_1(t) \| \, dt \right), 
$$
$$
 	\mathcal{G}(x) := 3 \mathcal{F}(x) (2 \| Q_1'(x) \| + \| Q_1^2(x) \| + \| Q_0(x) \| ),
$$
$$
 	\mathcal{H}(x) := \int_x^{\iy} \mathcal{G}(t) \mathcal{F}(t) \, dt.
$$
Note that $\mathcal{F}(x)$ is continuous and bounded, and by Lemma~\ref{lem:P}, 
$$
	\| P_{\pm}(x)\| \le \mathcal{F}(x), \quad \| P^*_{\pm}(t) \| \le \mathcal{F}(t). 
$$
Since the pencil $L$ belongs to $\mathcal{V}$, the function $\mathcal{G}(x)$ belongs to $L(0, \infty)$
and $\mathcal{H}(x)$ is continuous and bounded.

Now one can easily show that
$$
 	\begin{array}{ll}
	\| F_1(x, t, \rho) \| \le |\rho|^{-1} \mathcal{F}(x) \mathcal{G}(t), &
	\| F_2(x, t, \rho) \| \le |\rho|^{-2} \mathcal{F}(x) \mathcal{G}(t), \\
	\| F_3(x, t, \rho) \| \le             \mathcal{F}(x) \mathcal{G}(t), &
	\| F_4(x, t, \rho) \| \le |\rho|^{-1} \mathcal{F}(x) \mathcal{G}(t),	
 	\end{array}
$$
for $|\rho| \ge \rho_*$, $x \ge x_*$. Let us prove by induction the following estimates
\begin{equation} \label{estZk}
  	\| Z_k^{(\nu)}(x, \rho) \| \le \frac{2^k \mathcal{F}(x) \mathcal{H}^k(x)}{|\rho|^{k - \nu} k!}, \quad \nu = 0, 1, \, |\rho| \ge \rho_*, \, x \ge x_*.
\end{equation}
For $k = 0$ they are obvious. If they are already proved for some fixed $k$, for $k + 1$ we have
\begin{multline*}
 	\| Z_{k + 1}(x, \rho) \| \le \int_x^{\iy} \| F_1(x, t, \rho) \| \cdot \| Z_k(t) \| \, dt + 
 	\int_x^{\iy} \| F_2(x, t, \rho) \| \cdot \| Z_k'(t) \| \, dt \\
 	\le \frac{2^{k + 1} \mathcal{F}(x)}{|\rho|^{k + 1} k!} \int_x^{\iy} \mathcal{F}(t)\mathcal{G}(t) \mathcal{H}^k(t) \, dt = 
 	\frac{2^{k + 1} \mathcal{F}(x) \mathcal{H}^{k + 1}(x)}{|\rho|^{k + 1} (k + 1)!}.
\end{multline*}
The proof for $Z'_{k + 1}(x, \rho)$ is similar.

It follows from \eqref{estZk}, that the series \eqref{seriesZ} converges absolutely 
and uniformly for $|\rho| \ge \rho_*$, $x \ge x_*$, and the matrix function $Z(x, \rho)$ is 
the unique solution for the integral equation \eqref{inteqZ}. This function is continuous in $x$ 
and $\rho$ and analytic in $\rho$ for $|\rho| \ge \rho_*$, $x \ge x_*$.

For each fixed $x_* = \al > 0$ one can choose $\rho_{\al} = \rho_* > 0$ such that
\eqref{Q*} is satisfied. Then  
we construct $E_{+}(x, \rho)$ as the solution of \eqref{inteqJost}
for $x \ge \al$ and as the solution of the Cauchy problem for the equation $\ell_{\rho}(Y) = 0$, $x < \al$,
with initial conditions generated by $E_{+}(\al, \rho)$. Consequently, $E_{+}(x, \rho)$ satisfies $(i_1)$
and $(i_2)$.

Moreover, using \eqref{Z0}, \eqref{seriesZ} and \eqref{estZk}, we conclude
$$
 	\| Z(x, \rho) - P_{-}(x) \| \le (2 \mathcal{F}(x) \mathcal{H}(x) / |\rho|) \exp(2 \mathcal{H}(x) / |\rho|), 
$$
$$
  	\| Z'(x, \rho) + Q_1(x) P_{-}(x) \| \le (2 \mathcal{F}(x) \mathcal{H}(x)) \exp(2 \mathcal{H}(x) / |\rho|). 	
$$
The functions $\mathcal{F}(x)$ and $\mathcal{H}(x)$ are bounded and tend to zero as $x \to \infty$.
Therefore for each fixed $\rho$, $|\rho| \ge \rho_*$,
$$
 	Z(x, \rho) = P_{-}(x) + o(1), \quad Z'(x, \rho) = -Q_1(x) P_{-}(x) + o(1), \quad x \to \iy,
$$
that yields $(i_4)$, and for each fixed $x \ge 0$
$$
  	Z(x, \rho) = P_{-}(x) + O(|\rho|^{-1}), \quad Z'(x, \rho) = O(1), \quad |\rho| \to \iy.
$$
Substituting these asymptotics into \eqref{inteqZ}, we get
$$
 	Z(x, \rho) = P_{-}(x) - \frac{1}{2 i \rho} J_1(x)+ 
 	\frac{1}{2 i \rho}J_2(x, \rho) + O(|\rho|^{-2}), 
$$
where
$$
	J_1(x) := P_{-}(x) \int_x^{\iy} P_{+}^*(t) (Q_1'(t) - Q_1^2(t) - Q_0(t)) P_{-}(t) \, dt, 
$$
$$
 	J_2(x, \rho) := P_{+}(x) \int_x^{\infty} \exp(2 i \rho(t - x)) P_{-}^*(t) (Q_1'(t) - Q_1^2(t) - Q_0(t)) P_{-}(t) \, dt.
$$                          
Using \eqref{cauchyT}, \eqref{cauchyP*} and Lemma~\ref{lem:P}, we get 
\begin{multline*}
	J_1(x) = 2 i P_{-}(x) \int_x^{\infty} P_{+}^*(t) (T_{-}'(t) + Q_1(t) T_{-}(t)) \, dt =
	2 i P_{-}(x) \int_x^{\infty} (P_{+}^*(t) T_{-}'(t) + {P_{+}^*}'(t) T_{-}(t)) \, dt \\ =
	- 2 i P_{-}(x) P_{+}^*(x) T_{-}(x) = -2 i T_{-}(x).
\end{multline*}
Since 
$$
	P_{-}^*(t) (Q_1'(t) - Q_1^2(t) - Q_0(t)) \in L_2((0, \iy); \mathbb{C}^{m \times m}),
$$
we can apply 
\cite{FY01}[Lemma 2.1.1] to the entries of $J_2(x, \rho)$ and get
$$
 	\lim_{\rho \to \iy} \sup_{x \ge 0} \| J_2(x, \rho) \| = 0.
$$
Finally,
$$
 	Z(x, \rho) = P_{-}(x) + \frac{T_{-}(x)}{\rho} + o(\rho^{-1}), \quad x \ge \al, \quad |\rho| \to \infty,
$$
and now $(i_3)$ is obvious for $E_{+}^{(\nu)}(x, \rho)$.
 
{\bf Step 3.}
Let $\rho \in \Pi_{+}$, $\al \ge 0$ be fixed and $E_{-}(x, \rho)$ be the solution of the integral equation
\begin{multline*}
	E_{-}(x, \rho) = \exp(- i \rho x) P_{+}(x) + \frac{1}{2} \int_x^{\iy} \exp(i \rho (t - x)) P_{+}(x) P_{-}^*(t) 
	(i \rho I - Q_1(t))^{-1} F(t, \rho, E_{-}) \, dt \\
	 + \frac{1}{2} \int_{\al}^x \exp(i \rho(x - t)) P_{-}(x) P_{+}^*(t)
	(i \rho I - Q_1(t))^{-1} F(t, \rho, E_{-}) \, dt.
\end{multline*}
It is easy to check that $\ell_{\rho}(E_{-}) = 0$. Substituting $E_{-}(x, \rho) = \exp(- i \rho x) \xi(x, \rho)$,
we arrive at the equation
\begin{equation} \label{inteqxi}
\xi(x, \rho) = P_{+}(x) + \frac{1}{2} \int_x^{\iy} P_{+}(x) P_{-}^*(t) G(t, \rho, \xi) \, dt + 
\frac{1}{2} \int_a^x \exp(2 i \rho(x - t)) P_{-}(x) P_{+}^*(t) G(t, \rho, \xi) \, dt,
\end{equation}
where
$$
 	G(t, \rho, \xi) = (i \rho I - Q_1(t))^{-1} \bigl[ Q_1'(t) (i \rho I - Q_1(t))^{-1}
 	(\xi'(t, \rho) - i \rho \xi(t, \rho)) - (Q_1^2(t) + Q_0(t)) \xi(t, \rho)
 	\bigr].
$$
The method of successive approximations gives
$$
 	\xi_0(x, \rho) = P_{+}(x), \quad \xi_0'(x, \rho) = Q_1(x) P_{+}(x),
$$
\begin{multline*}
 	\xi_{k + 1}^{(\nu)}(x, \rho) = \int_x^{\iy} G_1^{(\nu)}(x, t, \rho) \xi_k(t, \rho) \, dt + 
						   \int_x^{\iy} G_2^{(\nu)}(x, t, \rho) \xi_k'(t, \rho) \, dt \\ +
						   \int_{\al}^x G_3^{(\nu)}(x, t, \rho) \xi_k(t, \rho) \, dt + 
						   \int_{\al}^x G_4^{(\nu)}(x, t, \rho) \xi_k'(t, \rho) \, dt, \quad \nu = 0, 1,
\end{multline*}
$$
 	\xi(x, \rho) = \sum_{k = 0}^{\iy} \xi_{k}(x, \rho).
$$
Here the functions $G_s(x, t, \rho)$, $s = \overline{1, 4}$, can be written in explicit form, and they are differentiated with respect to $x$.
Analogously to the Step 2, we obtain the following estimates
\begin{equation} \label{estGs}
	\begin{array}{l}
 	\| G_1^{(\nu)}(x, t, \rho) \|, \, \| G_3^{(\nu)}(x, t, \rho) \| \le |\rho|^{\nu - 1} \mathcal{F}(x) \mathcal{G}(t),
 	\\
 	\| G_2^{(\nu)}(x, t, \rho) \|, \, \| G_4^{(\nu)}(x, t, \rho) \| \le |\rho|^{\nu - 2} \mathcal{F}(x) \mathcal{G}(t),
 	\end{array}
\end{equation}
for $\nu = 0, 1$, $x \ge \al$, $\rho \ge \rho_*$, where $\rho^*$ is such that \eqref{Q*} is satisfied ($x_* = \al$).
By induction, we prove that
$$
 	\| \xi_k^{(\nu)}(x, \rho) \| \le \frac{2^k \mathcal{F}(x) \mathcal{H}^k(\al) }{|\rho|^{k - \nu}}, \quad \nu = 0, 1.
$$
Choose $\rho_{\al} > 0$ such that $\rho_{\al} \ge 4 \mathcal{H}(\al)$ and $\rho_{\al} \ge \rho_*$. Then 
for $|\rho| \ge \rho_{\al} $  and $x \ge \al$ the equation \eqref{inteqxi} is uniquely solvable,
and now $(i_1)$, $(i_2)$ for $E_{-}(x, \rho)$ can be proved similarly to $E_{+}(x, \rho)$.
Moreover,
\begin{equation} \label{estxi}
 	\| \xi^{(\nu)}(x, \rho) \| \le 2 |\rho|^{\nu} \mathcal{F}(x), \quad \nu = 0, 1,
\end{equation}
$$ 	
 	\| \xi(x, \rho) - P_{+}(x) \| \le 4 H(\al) / |\rho|, 
 	\quad \| \xi'(x, \rho) - Q_1(x) P_{+}(x) \| \le 4 H(\al).
$$
Consequently, for each fixed $x \ge \al$,
\begin{equation} \label{asymptxi}
 	\xi(x, \rho) = P_{+}(x) + O(\rho^{-1}), \quad \xi'(x, \rho) = O(1), \quad |\rho| \to \infty.
\end{equation}
If $\rho \in \Pi_{+}$ is fixed, there exists $\al \ge 0$ such that $|\rho| \ge 4 \mathcal{H}(\al)$
and \eqref{Q*} is satisfied. According to \eqref{inteqxi},
\begin{multline*}
 	\xi(x, \rho) = P_{+}(x) + \int_x^{\infty} G_1(x, t, \rho) \xi(t, \rho) \, dt + 
 	\int_x^{\infty} G_2(x, t, \rho) \xi'(t, \rho) \, dt \\ + \int_{\al}^{x} G_3(x, t, \rho) \xi(t, \rho) \, dt 
 	+ \int_{\al}^x G_4(x, t, \rho) \xi'(t, \rho) \, dt, 	
\end{multline*}
then using \eqref{estGs} and \eqref{estxi}, we derive
\begin{multline*}
 	\| \xi(x, \rho) - P_{+}(x) \| \le 4 \int_x^{\infty} \frac{\mathcal{F}(x) \mathcal{G}(t)}{|\rho|} \mathcal{F}(t) \, dt
 	+ 4 \int_{\al}^x \exp(2 i \rho (x - t)) \frac{\mathcal{F}(x) \mathcal{G}(t)}{|\rho|} \mathcal{F}(t) \, dt \\ \le
 	4 \int_x^{\infty} \frac{\mathcal{F}(x) \mathcal{G}(t)}{|\rho|} \mathcal{F}(t) \, dt
 	 + 4 \int_a^{x / 2} \exp(2 i \rho (x - t)) \frac{\mathcal{F}(x) \mathcal{G}(t)}{|\rho|} \mathcal{F}(t) \, dt
 	+ 4 \int_{x/2}^x \frac{\mathcal{F}(x) \mathcal{G}(t)}{|\rho|} \mathcal{F}(t) \, dt.
\end{multline*}
Finally,
$$
 	\| \xi(x, \rho) - P_{+}(x) \| \le 4 \mathcal{F}(x) \mathcal{H}(x/2) / |\rho| + 4 \exp(-|\mbox{Im} \, \rho|x) \mathcal{H}(a) \mathcal{F}(x)/ |\rho|.
$$
As $x \to \iy$, the first summand tends to zero since $\mathcal{H}(x/2)$ tends to zero and $\mathcal{F}(x)$ is bounded. The second summand 
tends to zero since $\mbox{Im} \, \rho \ne 0$. The analogous estimates can be obtained for $\xi'(x, \rho)$.
Therefore,
$$
 	\xi(x, \rho) = P_{+}(x) + o(1), \quad \xi'(x, \rho) = Q_1(x) P_{+}(x) + o(1), \quad x \to \infty,
$$                                                                   
so $(i_3)$ is proved for $E^{(\nu)}_{-}(x, \rho)$.

In order to prove $(i_4)$, we substitute \eqref{asymptxi} into \eqref{inteqxi}:
$$
	\xi(x, \rho) = P_{+}(x) - \frac{1}{2 i \rho} J_1(x) - \frac{1}{2 i \rho} J_2(\al, x, \rho) + O(\rho^{-2}),
$$                   
$$
 	J_1(x) := P_{+}(x) \int_x^{\infty} P_{-}^*(t) (Q_1'(t) + Q_1^2(t) + Q_0(t)) P_{+}(t) \, dt,
$$
$$
 	J_2(\al, x, \rho) := P_{-}(x) \int_{\al}^x \exp(2 i \rho (x - t)) P_{+}^*(t) (Q_1'(t) + Q_1^2(t) + Q_0(t)) P_{+}(t) \, dt.
$$                      
It is easy to check that $J_1(x) = -2 i T_{+}(x)$ ($T_{+}(x)$ was defined in \eqref{cauchyT}). Applying 
\cite{FY01}[Lemma 2.1.1] to the entries of $J_2(\al, x, \rho)$, for each fixed $\al$, we get
$$
 	\lim_{|\rho| \to \infty} \sup_{x \ge \al} \| J_2(\al, x, \rho) \| = 0.
$$
Consequently,
\begin{equation} \label{asymptxi2}
 	\xi(x, \rho) = P_{+}(x) + \frac{T_{+}(x)}{\rho} + o(\rho^{-1}), \quad |\rho| \to \infty.
\end{equation}
Using \eqref{asymptxi} and \eqref{asymptxi2}, we arrive at $(i_3)$ for $E^{(\nu)}_{-}(x, \rho)$.
                                                                              
{\bf Step 4.} Let us prove that the columns of $E_{+}(x, \rho)$ and $E_{-}(x, \rho)$ are linearly independent.
Fix $\rho \ne 0$, $\mbox{Im}\, \rho \ge 0$ and suppose this does not hold, i.e. there exist vectors $A(\rho)$ and $B(\rho)$, not both equal zero,
such that
$$
 	E^{(\nu)}_{-}(x, \rho) A(\rho) + E^{(\nu)}_{+}(x, \rho) B(\rho) = 0, \quad x \ge 0, \quad \nu = 0, 1.
$$
Substituting the asymptotics from $(i_4)$, we get
$$
	\begin{array}{l}
	\bigl[P_{+}(x) + o(1)\bigr] A(\rho) + \exp(2 i \rho x) \bigl[P_{-}(x) + o(1)\bigr] B(\rho) = 0, \\
	\bigl[- (i \rho I - Q_1(x)) P_{+}(x) + o(1) \bigr] A(x) + \exp(2 i \rho x) 
	\bigl[(i \rho I - Q_1(x))P_{-}(x) + o(1)\bigr] B(\rho) = 0, 
	\end{array}	
$$ 
as $x \to \infty$.
Since $Q_1(x)$ tends to zero as $x \to \infty$ and by Lemma~\ref{lem:P} the functions $P_{\pm}(x)$ are bounded,
we obtain 
$$
 	P_{+}(x) A(\rho) = o(1), \quad P_{-}(x) B(\rho) = o(1), \quad x \to \infty.
$$
Since the solutions of the Cauchy problems
$$
	Y_{\pm}(x) = \pm Q_1(x) Y_{\pm}(x), \quad Y_{\pm}(\infty) = 0, 	
$$
are unique: $Y_{\pm}(x) \equiv 0$, we conclude that $A(\rho) = B(\rho) = 0$. The contradiction finishes the proof of the theorem.
\end{proof}

\medskip



\vspace{1cm}

Natalia Bondarenko

Department of Mathematics

Saratov State University

Astrakhanskaya 83, Saratov 410026, Russia

{\it bondarenkonp@info.sgu.ru}

\smallskip

Gerhard Freiling

Department of Mathematics 

University Duisburg-Essen 

Forsthausweg 2, 47057 Duisburg, Germany 

{\it freiling@math.uni-duisburg.de}

\end{document}